\documentclass[a4paper,12pt,onecolumn]{article}

\usepackage{amsthm,amsmath,amscd}

\usepackage{cite}
\usepackage[charter]{mathdesign}
\usepackage[vmargin=2cm,hmargin=2cm,headheight=14.5pt,top=2cm,headsep=.5cm]{geometry}
\usepackage{cases}

\theoremstyle{theorem}

\newtheorem{thm}{Theorem}[section]

\newtheorem{lem}[thm]{Lemma}
\newtheorem{cor}[thm]{Corollary}

\theoremstyle{definition}

\newtheorem{dfn}[thm]{Definition}
\newtheorem{rem}[thm]{Remark}

\theoremstyle{remark}

\numberwithin{equation}{section}
\numberwithin{figure}{section}

\DeclareMathOperator{\supp}{supp}
\newcommand{\bR}{{\mathbb R}}

\renewcommand{\b}{\beta}

\renewcommand{\phi}{\varphi}

\newcommand{\ol}{\overline}
\newcommand{\fa}{\forall}
\newcommand{\nkp}{\enskip}
\newcommand{\sfa}{\nkp\fa}
\renewcommand{\(}{\left(}
\renewcommand{\)}{\right)}
\allowdisplaybreaks

\newenvironment{dedication}{\begin{quotation}\small\begin{em}}   {\par\end{em}\end{quotation}\vspace{1em}}

\begin{document}
\title{Nontrivial solutions of Hammerstein integral equations with reflections}
\date{}



\author{Alberto Cabada\footnote{alberto.cabada@usc.es. Departamento de An\'alise Ma\-te\-m\'a\-ti\-ca, Facultade de Matem\'aticas,
		Universidade de Santiago de Com\-pos\-te\-la, 15782 Santiago de Compostela, Spain. Partially supported by FEDER and Ministerio de Educaci\'on y Ciencia, Spain, project MTM2010-15314.}, \ Gennaro Infante\footnote{gennaro.infante@unical.it. Dipartimento di Matematica ed Informatica, Universit\`{a} della
		Calabria, 87036 Arcavacata di Rende, Cosenza, Italy. This paper was partially written during the visit of G. Infante to the
		Departamento de An\'alise Matem\'atica of the Universidade de Santiago
		de Compostela. G. Infante is grateful to the people of the
		aforementioned Departamento for their kind and warm hospitality.} \ and F. Adri\'an F. Tojo\footnote{fernandoadrian.fernandez@usc.es. F. Adri\'an F. Tojo, Departamento de An\'alise Ma\-te\-m\'a\-ti\-ca, Facultade de Matem\'aticas,
		Universidade de Santiago de Com\-pos\-te\-la, 15782 Santiago de Compostela, Spain. Partially supported by Diputaci\'on de A Coru\~na, Bolsas para Investigaci\'on 2012.}}

\maketitle

\begin{abstract}
Using the theory of fixed point index, we establish new results
for the existence of nonzero solutions of Hammerstein integral equations with reflections.
We apply our results to a first order periodic boundary value problem with reflections.
\end{abstract}

\begin{dedication}
It is a great pleasure for us to dedicate this paper to Professor Jean Mawhin on his seventieth birthday.
\end{dedication}
\section{Introduction}
In a recent paper Cabada and Tojo~\cite{alb-adr} studied, by means of methods and results present in \cite{Cab1,Cab2}, the first order operator $u'(t)+\omega\,u(-t)$ coupled with periodic boundary value conditions, describing the eigenvalues of the operator and providing the expression of the associated Green's function in the non-resonant case. One motivation for studying this particular problem is that differential equations with reflection of the argument have seen growing interest along the years, see for example the papers \cite{Aft, And, alb-adr, Gup, Gup2, Ma, Ore, Ore2, Pia, Pia2, Wie1} and references therein.
In \cite{alb-adr} the authors provide the range of values of the real parameter $\omega$ for which the Green's function has constant sign and apply these results to prove the existence of constant sign solutions for the nonlinear periodic problem with reflection of the argument
\begin{equation}\label{eqgenpro-intro}
u'(t)  =h(t,u(t),u(-t)),\,   t\in[-T,T],\quad
u(-T) =u(T).
\end{equation}
The methodology, analogous to the one utilized by Torres~\cite{Tor} in the case of ordinary differential equations, is to rewrite the problem \eqref{eqgenpro-intro} as an Hammerstein integral equation with reflections of the type
$$
u(t)=\int_{-T}^{T} k(t,s)[h(s,u(s),u(-s))+m\, u(-s)]\,ds,\quad t\in[-T,T],
$$
where the kernel $k$  has constant sign, and to make use of the well-known Guo-Krasnosel'ski\u\i{} theorem on
cone compression-expansion (see for example \cite{guolak}).

In this paper we continue the study of \cite{alb-adr} and we prove new results regarding the existence of nontrivial solutions of Hammerstein integral equations with reflections of the form
$$
u(t)=\int_{-T}^{T} k(t,s)g(s)f(s,u(s),u(-s))\,ds,\quad t\in[-T,T],
$$
where the kernel $k$ is allowed to be not of constant sign. In order to do this, we extend the results of \cite{gijwjiea}, valid for Hammerstein integral equations without reflections, to the new context. We make use of a cone of functions
that are allowed to change sign combined with the classical fixed point index for compact maps (we refer to \cite{amann} or \cite{guolak} for further information).
As an application of our theory we prove the existence of nontrivial solutions of the periodic problem with reflections~\eqref{eqgenpro-intro}.

\section{The case of kernels that change sign}
We begin with the case of kernels that are allowed to change sign.
We impose the following conditions on $k, f, g$ that occur in the integral equation
\begin{equation}\label{eqhamm}
u(t)=\int_{-T}^{T} k(t,s)g(s)f(s,u(s),u(-s))\,ds =:Fu(t),
\end{equation}
where $T$ is fixed in $(0,\infty)$.
\begin{enumerate}
\item [$(C_{1})$] The kernel $k$ is measurable, and for every $\tau\in
[-T,T]$ we have
\begin{equation*}
\lim_{t \to \tau} |k(t,s)-k(\tau,s)|=0 \;\text{ for almost every (a.\,e.) } s \in
[-T,T].
\end{equation*}{}
\item [$(C_{2})$]
 There exist a subinterval $[a,b] \subseteq [-T,T]$, a measurable function
$\Phi$ with $\Phi \geq 0$ a.\,e. and a constant $c=c(a,b) \in (0,1]$ such that
\begin{align*}
|k(t,s)|\leq \Phi(s) \text{ for  all }  &t \in [-T,T] \text{ and a.\,e. } \, s\in [-T,T],\\
k(t,s) \geq c\,\Phi(s) \text{ for  all } &t\in [a,b] \text{ and a.\,e. } \, s \in [-T,T].
\end{align*}{}
\item [ $(C_{3})$] The function $g$ satisfies that $g\,\Phi \in L^1[-T,T]$,  $g(t) \geq 0$ a.\,e. $t \in [-T,T]$
and $\int_a^b \Phi(s)g(s)\,ds >0$.{}
\item  [ $(C_{4})$] The nonlinearity $f:[-T,T]\times (-\infty,\infty) \times  (-\infty,\infty)
 \to [0,\infty)$ satisfies  the Carath\'{e}odory
conditions, that is, $f(\cdot,u,v)$ is measurable for each fixed
$u$ and $v$ and $f(t,\cdot,\cdot)$ is continuous for a.\,e. $t\in [-T,T]$, and for each $r>0$, there exists $\phi_{r} \in
L^{\infty}[-T,T]$ such that{}
\begin{equation*}
f(t,u,v)\le \phi_{r}(t) \;\text{ for all } \; (u,v)\in
[-r,r]\times [-r,r],\;\text{ and a.\,e. } \; t\in [-T,T].
\end{equation*}{}
\end{enumerate}
We recall the following definition.
\begin{dfn} Let $X$ be a Banach Space. A \textit{cone} on $X$ is a closed,
convex subset of $X$ such that $\lambda \, x\in K$ for $x \in K$ and
$\lambda\geq 0$ and $K\cap (-K)=\{0\}$.
\end{dfn}
Here we work in the space $C[-T,T]$, endowed with the usual supremum norm, and we use the cone
\begin{equation}\label{eqcone-cs}
K=\{u\in C[-T,T]: \min_{t \in [a,b]}u(t)\geq c \|u\|\},
\end{equation}
Note that $K \neq \{0\}$.

The cone $K$ has been essentially introduced by Infante and Webb in \cite{gijwjiea}
and later used in \cite{dfgior1, dfgior2, giems, gi-pp1, gi-pp2, gijwjmaa, gijwnodea, gijwems}. $K$ is similar to
a type of cone of \emph{non-negative} functions
 first used by Krasnosel'ski\u\i{}, see e.g. \cite{krzab}, and D.~Guo, see e.g. \cite{guolak}.
Note that functions in $K$ are positive on the
subset $[a,b]$ but are allowed to change sign in $[-T,T]$.

We require some knowledge of the classical fixed point index for
compact maps, see for example \cite{amann} or \cite{guolak} for
further information.
If $\Omega$ is a bounded open subset of $K$ (in the relative
topology) we denote by $\overline{\Omega}$ and $\partial \Omega$
the closure and the boundary relative to $K$. When $D$ is an open
bounded subset of $X$ we write $D_K=D \cap K$, an open subset of
$K$.

 Next Lemma is a direct consequence of classical results from degree theory.
\begin{lem} \label{lemind}
Let $D$ be an open bounded set with $D_{K}\ne \emptyset$ and
$\overline{D}_{K}\ne K$. Assume that $F:\overline{D}_{K}\to K$ is
a compact map such that $x\neq Fx$ for $x\in \partial D_{K}$. Then
the fixed point index $i_{K}(F, D_{K})$ has the following properties.
\begin{itemize}
\item[(1)] If there
exists $e\in K\setminus \{0\}$ such that $x\neq Fx+\lambda e$ for
all $x\in \partial D_K$ and all $\lambda
>0$, then $i_{K}(F, D_{K})=0$.
\item[(2)] If
$\mu x \neq Fx$
for all $x\in
\partial D_K$ and for every $\mu \geq 1$, then $i_{K}(F, D_{K})=1$.
\item[(3)] Let $D^{1}$ be open in $X$ with
$\overline{D^{1}}\subset D_K$. If $i_{K}(F, D_{K})=1$ and
$i_{K}(F, D_{K}^{1})=0$, then $F$ has a fixed point in
$D_{K}\setminus \overline{D_{K}^{1}}$. The same result holds if
$i_{K}(F, D_{K})=0$ and $i_{K}(F, D_{K}^{1})=1$.
\end{itemize}
\end{lem}

\begin{dfn}
We use the following sets: $$K_{\rho}=\{u\in K: \|u\|<\rho\},\
V_\rho=\{u \in K: \displaystyle{\min_{t\in [a,b]}}u(t)<\rho \}.$$ The set $V_\rho$ was introduced in \cite{gijwems}
and is equal to the set
called $\Omega_{\rho /c}$ in \cite{gijwjmaa}. The
notation $V_\rho$ makes it clear that
choosing $c$ as large as possible yields a weaker condition to be
satisfied by $f$ in Lemma \ref{idx0b1}.
A key feature
of these sets is that they can be
nested, that is
$$
K_{\rho}\subset V_{\rho}\subset K_{\rho/c}.
$$
\end{dfn}
\begin{thm}\label{thmk}
Assume that  hypotheses $(C_{1})$-$(C_{4})$ hold for some $r>0$. Then $F$ maps
$\overline{K}_{r}$ into $K$ and is compact. When these hypotheses
hold for each $r>0$, $F$ is compact and maps $K$ into $K$.
\end{thm}
\begin{proof}
For $u\in
\overline{K}_{r}\text{ and }t \in [-T,T]$ we have,
\begin{align*}
|Fu(t)|& \leq\int_{-T}^{T}
|k(t,s)|g(s)f(s,u(s),u(-s))\,ds\\
&\leq    \int_{-T}^{T} \Phi(s)g(s)f(s,u(s),u(-s))\,ds,
\end{align*}
and
$$
 \min_{t\in [a,b]}Fu(t) \geq+c
\int_{-T}^{T} \Phi(s)g(s)f(s,u(s),u(-s))\,ds\geq c\|Fu\|.
$$

Therefore we have that $Fu\in K$ for every $u\in
\overline{K}_{r}$.

The compactness of $F$ follows from the fact that the Hammerstein integral operator that occurs in \eqref{eqhamm} is compact (this a consequence
of Proposition 3.1 of Chapter 5
of~\cite{martin}).
\end{proof}
 In the sequel, we give a condition that ensures that, for a suitable $\rho>0$, the index is 1 on $K_{\rho}$.
\begin{lem}
\label{ind1b} Assume that
\begin{enumerate}
\item[$(\mathrm{I}_{\protect\rho }^{1})$] \label{EqB} there exists $\rho> 0$ such that
$$
 f^{-\rho,\rho}\cdot\sup_{t\in [-T,T]}\int_{-T}^{T}|k(t,s)|g(s)\,ds
 <1
$$
where
$$
  f^{{-\rho},{\rho}}=\sup \left\{\frac{f(t,u,v)}{\rho }:\;(t,u,v)\in
[ -T,T]\times [ -\rho,\rho ]\times [-\rho,\rho ]\right\}.$$ \end{enumerate}{}
Then the fixed point index, $i_{K}(F,K_{\rho})$, is equal to 1.
\end{lem}
\begin{proof}
We show that $\mu u \neq Fu$ for every $u \in \partial K_{\rho }$
and for every $\mu \geq 1$. 
In fact, if this does not happen, there exist $\mu \geq 1$ and $u\in
\partial K_{\rho }$ such that $\mu u=Fu$,
that is
$$\mu u(t)= \int_{-T}^{T}
k(t,s)g(s)f(s,u(s),u(-s))\,ds,$$
Taking the absolute value and then the supremum for $t\in [-T,T]$ gives
\begin{align*}
\mu \rho&\leq
\sup_{t\in [-T,T]}\int_{-T}^{T}|k(t,s)|g(s)f(s,u(s),u(-s))\,ds
 &\leq\rho f^{-\rho,\rho}\cdot\sup_{t\in [-T,T]}\int_{-T}^{T}|k(t,s)|g(s)\,ds
 <\rho.
\end{align*}

This
contradicts the fact that $\mu \geq 1$ and proves the result.
\end{proof}
\begin{rem}
We point out that, as in \cite{jwgi-nodea-08}, a stronger (but easier to check) condition than $(\mathrm{I}_{\protect\rho }^{1})$ is given by the following.
\begin{equation}
  \label{eqmest}
\frac{f^{{-\rho},{\rho}}}{m}<1,
\end{equation}
where
$$\frac{1}{m}:=\sup_{t\in [-T,T]}\int_{-T}^{T}|k(t,s)|g(s)\,ds.$$
\end{rem}

 Let us see now a condition that guarantees the index is equal to zero on $V_\rho$ for some appropriate $\rho>0$.
\begin{lem}
\label{idx0b1} Assume that

\begin{enumerate}
\item[$(\mathrm{I}_{\protect\rho }^{0})$] there exist $\rho >0$ such that
such that
$$
f_{(\rho ,{\rho /c})}\cdot\inf_{t\in [a,b]}\int_{a}^{b}k(t,s)g(s)\,ds
 >1,
$$
where
$$
f_{(\rho ,{\rho /c})} =\inf \left\{\frac{f(t,u,v)}{\rho }%
:\;(t,u,v)\in [a,b]\times [\rho ,\rho /c]\times
[-\rho /c,\rho /c]\right\}.$$
\end{enumerate}

Then $i_{K}(F,V_{\rho})=0$.
\end{lem}
\begin{proof} Let $e(t)\equiv 1$, then $e \in K$.
We prove that
\begin{equation*}
u\ne Fu+\lambda e\quad\text{for  all } u\in \partial
V_{\rho}\text{ and } \lambda \geq 0.
\end{equation*}
In fact, if not, there exist $u\in \partial V_\rho$ and $\lambda\geq 0$
such that $u=Fu+\lambda e$. Then we have$$
u(t)=\int_{-T}^{T}
k(t,s)g(s)f(s,u(s),u(-s))\,ds+\lambda.$$
Thus we get, for $t\in[a,b]$,

\begin{align*}
u(t)&=\int_{-T}^{T}
k(t,s)g(s)f(s,u(s),u(-s))\,ds+ \lambda \ge
\int_{a}^{b}
k(t,s)g(s)f(s,u(s),u(-s))\,ds\\ &\ge
\rho f_{(\rho ,{\rho /c})}
\;  \left(
\int_{a}^{b}
k(t,s)g(s)\,ds\right).
\end{align*}
Taking the minimum over $[a,b]$ gives
$\rho>\rho$ a contradiction.
\end{proof}
\begin{proof}
We prove that
\begin{equation*}
u\ne Fu\quad\text{for  all } u\in \partial
V_{\rho}.
\end{equation*}
In fact, if not, there exist $u\in \partial V_\rho$ such that $u=Fu$. Then we have$$
u(t)=\int_{-T}^{T}
k(t,s)g(s)f(s,u(s),u(-s))\,ds.$$
Thus we get, for $t\in[a,b]$,

\begin{align*}
u(t)&=\int_{-T}^{T}
k(t,s)g(s)f(s,u(s),u(-s))\,ds\ge
\int_{a}^{b}
k(t,s)g(s)f(s,u(s),u(-s))\,ds\\ &\ge
\rho f_{(\rho ,{\rho /c})}
\;  \left(
\int_{a}^{b}
k(t,s)g(s)\,ds\right).
\end{align*}
Taking the minimum over $[a,b]$ gives
$\rho>\rho$ a contradiction.
\end{proof}
\begin{rem}
We point out that, as in \cite{jwgi-nodea-08}, a stronger (but easier to check) condition than $(\mathrm{I}_{\protect\rho }^{0})$ is given by the following.
\begin{equation}
  \label{eqmest2}
\dfrac{f_{(\rho ,{\rho /c})}}{M(a,b)}>1,
\end{equation}
where
$$
\frac{1}{M(a,b)} :=\inf_{t\in [a,b]}\int_{a}^{b}k(t,s)g(s)\,ds.
$$
\end{rem}

The above Lemmas can be combined to prove the following Theorem. Here we
deal with the existence of at least one, two or three solutions. We stress
that, by expanding the lists in conditions $(S_{5}),(S_{6})$ below, it is
possible to state results for four or more positive solutions, see for
example the paper by Lan~\cite{kljdeds} for the type of results that might be stated. We omit
the proof which follows directly from the properties of the fixed point index  stated in Lemma \ref{lemind}, (3).
\begin{thm}
\label{thmmsol1} The integral equation \eqref{eqhamm} has at least one non-zero solution
in $K$ if either of the following conditions hold.

\begin{enumerate}

\item[$(S_{1})$] There exist $\rho _{1},\rho _{2}\in (0,\infty )$ with $\rho
_{1}/c<\rho _{2}$ such that $(\mathrm{I}_{\rho _{1}}^{0})$ and $(\mathrm{I}_{\rho _{2}}^{1})$ hold.

\item[$(S_{2})$] There exist $\rho _{1},\rho _{2}\in (0,\infty )$ with $\rho
_{1}<\rho _{2}$ such that $(\mathrm{I}_{\rho _{1}}^{1})$ and $(\mathrm{I}%
_{\rho _{2}}^{0})$ hold.
\end{enumerate}
The integral equation \eqref{eqhamm} has at least two non-zero solutions in $K$ if one of
the following conditions hold.

\begin{enumerate}

\item[$(S_{3})$] There exist $\rho _{1},\rho _{2},\rho _{3}\in (0,\infty )$
with $\rho _{1}/c<\rho _{2}<\rho _{3}$ such that $(\mathrm{I}_{\rho
_{1}}^{0}),$ $(
\mathrm{I}_{\rho _{2}}^{1})$ $\text{and}\;\;(\mathrm{I}_{\rho _{3}}^{0})$
hold.

\item[$(S_{4})$] There exist $\rho _{1},\rho _{2},\rho _{3}\in (0,\infty )$
with $\rho _{1}<\rho _{2}$ and $\rho _{2}/c<\rho _{3}$ such that $(\mathrm{I}%
_{\rho _{1}}^{1}),\;\;(\mathrm{I}_{\rho _{2}}^{0})$ $\text{and}\;\;(\mathrm{I%
}_{\rho _{3}}^{1})$ hold.
\end{enumerate}
The integral equation \eqref{eqhamm} has at least three non-zero solutions in $K$ if one
of the following conditions hold.

\begin{enumerate}
\item[$(S_{5})$] There exist $\rho _{1},\rho _{2},\rho _{3},\rho _{4}\in
(0,\infty )$ with $\rho _{1}/c<\rho _{2}<\rho _{3}$ and $\rho _{3}/c<\rho
_{4}$ such that $(\mathrm{I}_{\rho _{1}}^{0}),$ $(\mathrm{I}_{\rho _{2}}^{1}),\;\;(\mathrm{I}%
_{\rho _{3}}^{0})\;\;\text{and}\;\;(\mathrm{I}_{\rho _{4}}^{1})$ hold.

\item[$(S_{6})$] There exist $\rho _{1},\rho _{2},\rho _{3},\rho _{4}\in
(0,\infty )$ with $\rho _{1}<\rho _{2}$ and $\rho _{2}/c<\rho _{3}<\rho _{4}$
such that $(\mathrm{I}_{\rho _{1}}^{1})$, $(\mathrm{I}_{\rho
_{2}}^{0})$, $(\mathrm{I}_{\rho _{3}}^{1})$ and $(\mathrm{I}%
_{\rho _{4}}^{0})$ hold.
\end{enumerate}
\end{thm}
\section{The case of non-negative kernels}
We now assume the functions $k, f, g$ that occur in \eqref{eqhamm} satisfy the conditions $(C_{1})-(C_{4})$ in the previous section,
where $(C_{2})$ and $(C_{4})$ are replaced with the following.
\begin{enumerate}
\item [$(C_{2}')$] The kernel $k$ is non-negative for $t \in [-T,T] \text{ and a.\,e. } \, s\in [-T,T]$ and
there exist a subinterval $[a,b] \subseteq [-T,T]$, a measurable function
$\Phi$, and a constant $c=c(a,b) \in (0,1]$ such that
\begin{align*}
k(t,s)\leq \Phi(s) \text{ for } &t \in [-T,T] \text{ and a.\,e. } \, s\in [-T,T],\\
k(t,s) \geq c\Phi(s) \text{ for } &t\in [a,b] \text{ and a.\,e. } \, s \in [-T,T].
\end{align*}{}
\item [$(C_{4}')$] The nonlinearity $f:[-T,T]\times [0,\infty) \times  [0,\infty)
 \to [0,\infty)$ satisfies Carath\'{e}odory
conditions, that is, $f(\cdot,u,v)$ is measurable for each fixed
$u$ and $v$ and $f(t,\cdot,\cdot)$ is continuous for a.\,e. $t\in [-T,T]$, and for each $r>0$, there exists $\phi_{r} \in
L^{\infty}[-T,T]$ such that{}
\begin{equation*}
f(t,u,v)\le \phi_{r}(t) \;\text{ for all } \; (u,v)\in
[0,r]\times [0,r],\;\text{ and a.\,e. } \; t\in [-T,T].
\end{equation*}{}
\end{enumerate}
These hypotheses enable us to work in the cone of non-negative functions
\begin{equation}\label{eqcone-ng}
K'=\{u\in C[-T,T]: u\geq 0, \min_{t \in [a,b]}u(t)\geq c \|u\|\},
\end{equation}
that is smaller than the cone \eqref{eqcone-cs}. It is possible to show that $F$ is compact and leaves the cone $K'$ invariant.
The conditions on the index are given by the following Lemmas, the proofs are omitted as they are similar to the ones in the previous section.
\begin{lem}
\label{ind1-ng} Assume that
\begin{enumerate}
\item[$\(\ol{\mathrm{I}_{\protect\rho}^{1}}\)$] \label{EqB-ng} there exists $\rho> 0$ such that
$
  f^{0,\rho} < m,
$
where
$$
  f^{{0},{\rho}}=\sup \left\{\frac{f(t,u,v)}{\rho }:\;(t,u,v)\in
[ -T,T]\times [0,\rho ]\times [0,\rho ]\right\}.$$
\end{enumerate}{}
Then $i_{K}(F,K_{\rho})=1$.
\end{lem}
\begin{lem}
\label{idx0-ng} Assume that
\begin{enumerate}
\item[$\(\ol{\mathrm{I}_{\protect\rho }^{0}}\)$] there exist $\rho >0$ such that
$
f_{(\rho ,{\rho /c})'}>M
$,  where
$$
f_{(\rho ,{\rho /c})'} =\inf \left\{\frac{f(t,u,v)}{\rho }%
:\;(t,u,v)\in [a,b]\times [\rho ,\rho /c]\times
[0,\rho /c]\right\}.$$
\end{enumerate}
Then $i_{K}(F,V_{\rho})=0$.
\end{lem}
A result equivalent to Theorem~\ref{thmmsol1} is clearly valid in this case, with nontrivial solutions belonging to the cone \eqref{eqcone-ng}.
\section{The case of kernels with extra positivity}
We now assume the the functions $k, f, g$ that occur in \eqref{eqhamm} satisfy the conditions $(C_{1})$,$(C_{2}')$, $(C_{3})$ and $(C_{4}')$ with $[a,b]=[-T,T]$; in particular note that the kernel satisfies the stronger positivity requirement
$$
c\Phi(s) \leq k(t,s)\leq \Phi(s) \text{ for } t \in [-T,T] \text{ and a.\,e. } \, s\in [-T,T].
$$
These hypotheses enable us to work in the cone
\begin{equation}\label{eqcone-sp}
K^{''}=\{u\in C[-T,T]: \min_{t \in [-T,-T]}u(t)\geq c \|u\|\}.
\end{equation}
\begin{rem}
Note that a function in $K''$ that possesses a  non-trivial norm, has the useful property that is strictly positive on $[-T,T]$.
\end{rem}
Once gain $F$ is compact and leaves the cone $K''$ invariant. The assumptions on the index are as follows.
\begin{lem}
\label{ind1-sp} Assume that
\begin{enumerate}
\item[$\(\widetilde{\mathrm{I}_{\protect\rho}^{1}}\)$] \label{EqB-sp} there exists $\rho> 0$ such that
$
  f^{c\rho,\rho} < m,
$
where
$$
  f^{c\rho,{\rho}}=\sup \left\{\frac{f(t,u,v)}{\rho }:\;(t,u,v)\in
[ -T,T]\times [c\rho,\rho ]\times [c\rho,\rho ]\right\}.$$
\end{enumerate}{}
Then $i_{K}(F,K_{\rho})=1$.
\end{lem}
\begin{lem}
\label{idx0-sp} Assume that
\begin{enumerate}
\item[$\(\widetilde{\mathrm{I}_{\protect\rho}^{1}}\)$] there exist $\rho >0$ such that
$
f_{(\rho ,{\rho /c})^{''}}>M
$,  where
$$
f_{(\rho ,{\rho /c})^{''}}=\inf \left\{\frac{f(t,u,v)}{\rho }%
:\;(t,u,v)\in [a,b]\times [\rho ,\rho /c]\times
[\rho,\rho /c]\right\}.$$
\end{enumerate}
Then $i_{K}(F,V_{\rho})=0$.
\end{lem}
A result similar to Theorem~~\ref{thmmsol1} holds in this case.
\begin{rem}
If $f$ is defined only on $[-T,T]\times[u_1,u_2]\times [v_1,v_2]$ we can extend it to $[-T,T]\times\bR\times\bR$ considering firstly
$$
\bar{f}(t,u,v):=
\begin{cases}
f(t,u_{1},v),\ \text{if}\ 0 \leq u \leq u_{1},\\
f(t,u,v),\;\ \text{if} \ u_{1}\leq u \leq u_{2}, \\
f(t,u_{2},v),\ \text{if} \ u_{2}\leq u < \infty,
\end{cases}
$$
and secondly
$$
\tilde{f}(t,u,v):=
\begin{cases}
\bar{f}(t,u,v_{1}),\ \text{if}\ 0 \leq v \leq v_{1},\\
\bar{f}(t,u,v),\;\ \text{if} \ v_{1}\leq v \leq v_{2}, \\
\bar{f}(t,u,v_{2}),\ \text{if} \ v_{2}\leq v < \infty.
\end{cases}
$$
\end{rem}
\begin{rem}
Note that results similar to Sections 2, 3 and 4 hold when the kernel $k$ is negative on a strip, negative and strictly negative. This gives nontrivial solutions that are negative on an interval, negative and strictly negative respectively.
\end{rem}
\section{An application}
We now turn our attention to the first order functional periodic boundary value problem
\begin{equation}\label{eqgenpro2}
u'(t)  =h(t,u(t),u(-t)),\,  t\in [-T,T],
\end{equation}
\begin{equation}\label{prdic}
u(-T)=u(T),
\end{equation}

We apply the shift argument of \cite{alb-adr} (a similar idea has been used in \cite{Tor,jwmz-na}), by
fixing $\omega \in \mathbb{R}\setminus \{0\}$ and considering the  equivalent expression
\begin{equation}\label{eqgenpro2b}
u'(t) +\omega u(-t) =h(t,u(t),u(-t))+\omega u(-t):=f(t,u(t),u(-t)),\,  t\in [-T,T],
\end{equation}
\begin{equation}\label{prdicb}
u(-T)=u(T).
\end{equation}
 Following the ideas developed in \cite{alb-adr}, we can verify that the functional boundary value problem \eqref{eqgenpro2b}-\eqref{prdicb} can be rewritten into a Hammerstein integral equation of the type
\begin{equation}\label{gensol}
u(t)=\int_{-T}^{T} k(t,s)f(s,u(s),u(-s))\,ds,
\end{equation}
Also, $k(t,s)$ can be expressed in the following way  (see \cite{alb-adr} for details):
\begin{equation}\label{gbarra}
2\sin(\omega T)k(t,s)=\begin{cases} \cos \omega (T-s-t)+\sin \omega (T+s-t) , &  t>|s|,\\\cos \omega (T-s-t)-\sin \omega (T-s+t) , &  |t|<s,\\
\cos \omega (T+s+t)+\sin \omega(T+s-t) , &  |t|<-s,\\
\cos \omega (T+s+t)-\sin \omega(T-s+t) , &  t<-|s|.\end{cases}
\end{equation}
The results that follow are meant to prove that we are under the hypothesis of Theorem \ref{thmk}.\par
The sign properties of the kernel \eqref{gbarra} can be summarized as follows:
\begin{thm}\cite{alb-adr}\label{alphasign} Let $\zeta=\omega T$.
\begin{enumerate}
\item If $\zeta\in(0,\frac{\pi}{4})$ then $k(t,s)$ is strictly positive on $[-T,T]^2$.
\item If $\zeta\in(-\frac{\pi}{4},0)$ then $k(t,s)$ is strictly negative on $[-T,T]^2$.
\item If $\zeta=\frac{\pi}{4}$ then $k(t,s)$ vanishes on ${P}:=\{(-T,-T),(0,0),(T,T),(T,-T)\}$ and is strictly positive on $([-T,T]^2)\backslash {P}$.
\item If $\zeta=-\frac{\pi}{4}$ then $k(t,s)$ vanishes on ${P}$ and is strictly negative on $([-T,T]^2)\setminus {P}$.
\item If $\zeta\in\mathbb{R}\setminus[-\frac{\pi}{4},\frac{\pi}{4}]$ then $k(t,s)$ is changes sign on $[-T,T]^2$.
\end{enumerate}
\end{thm}

 In \cite{alb-adr} some existence results has been obtained for problem \eqref{eqgenpro2b}-\eqref{prdicb} when $\zeta\in[-\frac{\pi}{4},\frac{\pi}{4}]$, i.e., when the kernel $k$ has constant sign on $[-T,T]^2$. But nothing is obtained for the changing sign case.
Still, there are some things to be said about the kernel $k$ when $\zeta\in\mathbb{R}\setminus[-\frac{\pi}{4},\frac{\pi}{4}]$. First, realize that, using the trigonometric identities
$\cos(a-b)\pm\sin(a+b)=(\cos a\pm\sin a)(\cos b\pm\sin b)$ and $\cos(a)+\sin(a)=\sqrt{2}\cos(a-\frac{\pi}{4})$ and making the change of variables $t=Tz$, $s=Ty$, we can express $k$ as
\begin{equation}\label{eqgb2}
\sin(\zeta) k(z,y)=
\begin{cases}
\cos[\zeta(1-z)-\frac{\pi}{4}]\cos(\zeta y-\frac{\pi}{4}), &  z>|y|, \\
\cos(\zeta z+\frac{\pi}{4})\cos[\zeta(y-1)-\frac{\pi}{4}] , &  |z|<y, \\
\cos(\zeta z+\frac{\pi}{4})\cos[\zeta(1+y)-\frac{\pi}{4}], &  -|z|>y, \\
\cos[\zeta(z+1)+\frac{\pi}{4}]\cos(\zeta y-\frac{\pi}{4}), &  z<-|y|.
\end{cases}
\end{equation}
The following lemma relates the sign of $k$ for $\zeta$ positive and negative.
\begin{lem}\label{Gop}\cite{alb-adr} $k_\zeta(t,s)=-k_{-\zeta}(-t,-s)\sfa \, t,s\in I$ where $k_\zeta$ is the kernel for the value $\zeta$.
\end{lem}
Now we have the following result.
\begin{lem}\label{posstrip} The following hold:\par
\begin{enumerate}
\item If $\zeta\in(\frac{\pi}{4},\frac{\pi}{2})$, then $k$ is strictly positive in $$S:=
\left[\(-\frac{\pi}{4|\zeta|},\frac{\pi}{4|\zeta|}-1\)\cup\(1-\frac{\pi}{4|\zeta|},\frac{\pi}{4|\zeta|}\)\right]\times[-1,1].$$
\item If $\zeta\in(-\frac{\pi}{2},-\frac{\pi}{4})$, $k$ is strictly negative in $S$.
\end{enumerate}
\end{lem}
\begin{proof} By Lemma \ref{Gop}, it is enough to prove that $k$ is strictly positive in $S$ for $\zeta\in(\frac{\pi}{4},\frac{\pi}{2})$. We do here the proof for the connected component $\(1-\frac{\pi}{4\zeta},\frac{\pi}{4\zeta}\)\times[-1,1]$ of $S$. For the other one the proof is analogous.\par
If $z\in\(1-\frac{\pi}{4\zeta},\frac{\pi}{4\zeta}\)$, then $\zeta z+\frac{\pi}{4}\in\(\zeta,\frac{\pi}{2}\)\subset\(\frac{\pi}{4},\frac{\pi}{2}\)$, and hence $\cos\(\zeta z+\frac{\pi}{4}\)>0$.\par
Also, if $z\in\(1-\frac{\pi}{4\zeta},\frac{\pi}{4\zeta}\)$, then $\zeta(1-z)-\frac{\pi}{4}\in\(\zeta-\frac{\pi}{2},0\)\subset\(-\frac{\pi}{4},0\)$ and therefore \\ $\cos\(\zeta(1-z)-\frac{\pi}{4}\)>0$.\par
If $y\in\(-\frac{\pi}{4\zeta},\frac{\pi}{4\zeta}\)$, then $\zeta y-\frac{\pi}{4}\in\(-\frac{\pi}{2},0\)$ so $\cos\(\zeta y-\frac{\pi}{4}\)>0$.\par
If $y\in \(1-\frac{\pi}{4\zeta},1\)$, then $\zeta(y-1)-\frac{\pi}{4}\in\(-\frac{\pi}{2},-\frac{\pi}{4}\)$ so $\cos\(\zeta(y-1)-\frac{\pi}{4}\)>0$.\par
If $y\in\(-1,\frac{\pi}{4\zeta}-1\)$, then $\zeta(y+1)+\frac{\pi}{4}\in\(\frac{\pi}{4},\frac{\pi}{2}\)$ so $\cos\(\zeta(y+1)+\frac{\pi}{4}\)>0$.\par
With these inequalities the result is straightforward from equation (\ref{eqgb2}).
\end{proof}
\begin{lem}\label{max1}If $\zeta\in(\frac{\pi}{4},\frac{\pi}{2})$ then
$\sin(\zeta) |k(z,y)|\le\Phi(y):=\sin(\zeta)\,\max_{z\in[-1,1]}k(z,y)$ where $\Phi$ admits the following expression:
$$\Phi(y)=\begin{cases}
\cos\left[\zeta(y-1)-\frac{\pi}{4}\right] , &  y\in[\b,1], \\
\cos\left[\zeta(y-1)+\frac{\pi}{4}\right]\cos\(\zeta y-\frac{\pi}{4}\) , &  y\in\left[1-\frac{\pi}{4\zeta},\b\), \\
\cos\(\zeta y-\frac{\pi}{4}\) , &  y\in\left[\b-1,1-\frac{\pi}{4\zeta}\), \\
\cos\(\zeta y+\frac{\pi}{4}\)\cos\left[\zeta(y+1)-\frac{\pi}{4}\right] , &  y\in[-\frac{\pi}{4\zeta},\b-1), \\
\cos\left[\zeta(y+1)-\frac{\pi}{4}\right] , &  y\in[-1,-\frac{\pi}{4\zeta}).
\end{cases}
$$
where $\b$ is the only solution of the equation
\begin{equation}\label{eqbeta}\cos\left[\zeta(y-1)+\frac{\pi}{4}\right]\cos\(\zeta y -\frac{\pi}{4}\)-\cos\left[\zeta(y-1)-\frac{\pi}{4}\right]=0
\end{equation}
in the interval $\left[\frac{1}{2},1\right]$.

\end{lem}
\begin{proof} Let
$$v(y):=\cos\left[\zeta(y-1)+\frac{\pi}{4}\right]\cos\(\zeta y -\frac{\pi}{4}\)-\cos\left[\zeta(y-1)-\frac{\pi}{4}\right],$$
then
$$v'(y)=\zeta\left[\sin\(\zeta(y-1)-\frac{\pi}{4}\)-\sin\(\zeta(2y-1)\)\right].$$
Observe that $y\in\left[\frac{1}{2},1\right]$ implies $\zeta(y-1)-\frac{\pi}{4}\in\left[-\frac{\zeta}{2}-\frac{\pi}{4},-\frac{\pi}{4}\right]\subset\left[-\frac{3\pi}{4},-\frac{\pi}{4}\right]$ and $\zeta(2y-1)\in(0,\zeta)\subset\left[0,\frac{\pi}{2}\right]$, therefore $v'(y)\le0\sfa y\in\left[\frac{1}{2},1\right]$. Furthermore, since $\zeta\in(\frac{\pi}{4},\frac{\pi}{2})$, \begin{align*}
v\(\frac{1}{2}\) & =\cos^2\(\frac{\zeta}{2}-\frac{\pi}{4}\)-\cos\(\frac{\zeta}{2}+\frac{\pi}{4}\)=1-\left[\cos\(-\frac{\zeta}{2}\)+\frac{\sqrt 2}{2}\right]\left[\sin\(-\frac{\zeta}{2}\)+\frac{\sqrt 2}{2}\right] \\ & \ge\frac{\sqrt{4-2\sqrt{2}}}{2}>0, \\ v(1) & =\frac{\sqrt 2}{2}\left[1-\cos\(\zeta-\frac{\pi}{4}\)\right]\le0.
\end{align*}
Hence, equation (\ref{eqbeta}) has a unique solution $\b$ in $\left[\frac{1}{2},1\right]$. Besides, since $v(\frac{\pi}{4\zeta})=\sqrt 2\sin(\zeta-\frac{\pi}{4})>0$, we have that $\b>\frac{\pi}{4\zeta}$. Furthermore, is easy to check that $$-1<-\frac{\pi}{4\zeta}<\b-1<\frac{\pi}{4\zeta}-1<0<1-\frac{\pi}{4\zeta}<\frac{\pi}{4\zeta}<\b<1.$$\par
Now, realize that
\begin{equation}\label{firstineq}
\sin(\zeta) k(z,y)\le\xi(z,y):=
\begin{cases}
\cos[\zeta(1-\max\{1-\frac{\pi}{4\zeta},|y|\})-\frac{\pi}{4}]\cos(\zeta y-\frac{\pi}{4}), &  z>|y|, \\
\cos(\zeta\min\{\frac{\pi}{4\zeta},y\}-\frac{\pi}{4})\cos[\zeta(y-1)-\frac{\pi}{4}] , &  |z|<y, \\
\cos(\zeta\max\{-\frac{\pi}{4\zeta},y\}+\frac{\pi}{4})\cos[\zeta(1+y)-\frac{\pi}{4}], &  -|z|>y, \\
\frac{\sqrt 2}{2}\cos(\zeta y-\frac{\pi}{4}), &  z<-|y|,
\end{cases}
\end{equation}
while $\xi(z,y)\le\Phi(y)$.

We study now the different cases for the value of $y$.\newline\par
$\bullet$ If $y\in[\b,1]$, then

\begin{subnumcases}{\label{xi1} \xi(z,y)=}
  \label{xi1a} \cos\left[\zeta(y-1)+\frac{\pi}{4}\right]\cos\(\zeta y-\frac{\pi}{4}\), & $z>y,$  \\
 \label{xi1b} \cos\left[\zeta(y-1)-\frac{\pi}{4}\right], & $ |z|<y,$ \\
 \label{xi1c} \frac{\sqrt 2}{2}\cos\(\zeta y-\frac{\pi}{4}\), & $ z<-y.$
 \end{subnumcases}

It is straightforward that $\cos[\zeta(y-1)+\frac{\pi}{4}]>\cos(\frac{\pi}{4})=\frac{\sqrt 2}{2}$, so (\ref{xi1a})$>$(\ref{xi1c}). By our study of equation (\ref{eqbeta}), it is clear that
$$\cos\left[\zeta(y-1)+\frac{\pi}{4}\right]\cos\(\zeta y-\frac{\pi}{4}\)\le\cos\left[\zeta(y-1)-\frac{\pi}{4}\right].$$
Therefore (\ref{xi1a})$\ge$(\ref{xi1b}) and $\Phi(y)=\cos\left[\zeta(y-1)-\frac{\pi}{4}\right]$.\newline\par

$\bullet$ If $y\in\left[\frac{\pi}{4\zeta},\b\)$, then $\xi$ is as in (\ref{xi1}) and (\ref{xi1a})$>$(\ref{xi1c}), but in this case
$$\cos\left[\zeta(y-1)+\frac{\pi}{4}\right]\cos\(\zeta y-\frac{\pi}{4}\)\ge\cos\left[\zeta(y-1)-\frac{\pi}{4}\right],$$
so  (\ref{xi1a})$\le$(\ref{xi1b}) and $\Phi(y)=\cos\left[\zeta(y-1)+\frac{\pi}{4}\right]\cos\(\zeta y-\frac{\pi}{4}\)$.\newline\par

$\bullet$ If $y\in\left[1-\frac{\pi}{4\zeta},\frac{\pi}{4\zeta}\)$, then

\begin{subnumcases}{\label{xi2} \xi(z,y)=}
  \label{xi2a} \cos\left[\zeta(y-1)+\frac{\pi}{4}\right]\cos\(\zeta y-\frac{\pi}{4}\), & $z>y,$  \\
 \label{xi2b} \cos\left[\zeta(y-1)-\frac{\pi}{4}\right]\cos\(\zeta y-\frac{\pi}{4}\), & $ |z|<y,$ \\
 \label{xi2c} \frac{\sqrt 2}{2}\cos\(\zeta y-\frac{\pi}{4}\), & $ z<-y.$
 \end{subnumcases}
We have that
$$\cos\left[\zeta(y-1)+\frac{\pi}{4}\right]-\cos\left[\zeta(y-1)-\frac{\pi}{4}\right]=\sqrt 2\sin[\zeta(1-y)]>0,$$
therefore (\ref{xi2a})$\ge$(\ref{xi2b}) and $\Phi(y)=\cos[\zeta(y-1)+\frac{\pi}{4}]\cos(\zeta y-\frac{\pi}{4})$.\newline\par

$\bullet$ If $y\in\left[0,1-\frac{\pi}{4\zeta}\)$, then
\begin{subnumcases}{\label{xi3} \xi(z,y)=}
  \label{xi3a} \cos\(\zeta y-\frac{\pi}{4}\), & $z>y,$  \\
 \label{xi3b} \cos\left[\zeta(y-1)-\frac{\pi}{4}\right]\cos\(\zeta y-\frac{\pi}{4}\), & $ |z|<y,$ \\
 \label{xi3c} \frac{\sqrt 2}{2}\cos\(\zeta y-\frac{\pi}{4}\), & $ z<-y.$
 \end{subnumcases}
 $\cos\left[\zeta(y-1)-\frac{\pi}{4}\right]<\frac{\sqrt 2}{2}$, so (\ref{xi3b})$\le$(\ref{xi3c})$\le$(\ref{xi3a}) and $\Phi(y)=\cos\(\zeta y-\frac{\pi}{4}\)$.\newline\par

 $\bullet$ If $y\in\left[\b-1,0\)$, then

\begin{subnumcases}{\label{xi4} \xi(z,y)=}
  \label{xi4a} \cos\(\zeta y-\frac{\pi}{4}\), & $z>-y,$  \\
 \label{xi4b} \cos\(\zeta y+\frac{\pi}{4}\)\cos\left[\zeta(1+y)-\frac{\pi}{4}\right], & $ -|z|>y,$ \\
 \label{xi4c} \frac{\sqrt 2}{2}\cos\(\zeta y-\frac{\pi}{4}\), & $ z<y.$
 \end{subnumcases}
Let $y=\ol y-1$, then
$$\cos\(\zeta y+\frac{\pi}{4}\)\cos\left[\zeta(1+y)-\frac{\pi}{4}\right]\le \cos\(\zeta y-\frac{\pi}{4}\)$$ if and only if $$\cos\left[\zeta (\ol y-1)+\frac{\pi}{4}\right]\cos\(\zeta\ol y-\frac{\pi}{4}\)\le \cos\left[\zeta(\ol y-1)-\frac{\pi}{4}\right]$$
which is true as $\ol y\in[\b,1)$ and our study of equation (\ref{eqbeta}). Hence, $\Phi(y)=\cos\(\zeta y-\frac{\pi}{4}\)$.\newline\par

 $\bullet$ If $y\in\left[\frac{\pi}{4\zeta}-1,\b-1\)$, then

$\xi$ is the same as in (\ref{xi4}) but in this case $$\cos\(\zeta y+\frac{\pi}{4}\)\cos\left[\zeta(1+y)-\frac{\pi}{4}\right]\ge \cos\(\zeta y-\frac{\pi}{4}\)$$ so $\Phi(y)=\cos\(\zeta y+\frac{\pi}{4}\)\cos\left[\zeta(1+y)-\frac{\pi}{4}\right]$.\newline\par

 $\bullet$ If $y\in\left[-\frac{\pi}{4\zeta},\frac{\pi}{4\zeta}-1\)$, then
\begin{subnumcases}{\label{xi5} \xi(z,y)=}
  \label{xi5a} \cos\left[\zeta(1-y)-\frac{\pi}{4}\right]\cos\(\zeta y-\frac{\pi}{4}\), & $z>-y,$  \\
 \label{xi5b} \cos\(\zeta y+\frac{\pi}{4}\)\cos\left[\zeta(1+y)-\frac{\pi}{4}\right], & $ -|z|>y,$ \\
 \label{xi5c} \frac{\sqrt 2}{2}\cos\(\zeta y-\frac{\pi}{4}\), & $ z<y.$
 \end{subnumcases}
$$\cos\(\zeta y+\frac{\pi}{4}\)\cos\left[\zeta(1+y)-\frac{\pi}{4}\right]-\cos\left[\zeta(1-y)-\frac{\pi}{4}\right]\cos\(\zeta y-\frac{\pi}{4}\)\\ =-\sin\zeta\sin(2\zeta y)>0,$$
then $\Phi(y)=\cos\(\zeta y+\frac{\pi}{4}\)\cos\left[\zeta(1+y)-\frac{\pi}{4}\right]$.\newline\par

 $\bullet$ If $y\in\left[-1,-\frac{\pi}{4\zeta}\)$, then
\begin{subnumcases}{\label{xi6} \xi(z,y)=}
  \label{xi6a} \cos\left[\zeta(1-y)-\frac{\pi}{4}\right]\cos\(\zeta y-\frac{\pi}{4}\), & $z>-y,$  \\
 \label{xi6b} \cos\left[\zeta(1+y)-\frac{\pi}{4}\right], & $ -|z|>y,$ \\
 \label{xi6c} \frac{\sqrt 2}{2}\cos\(\zeta y-\frac{\pi}{4}\), & $ z<y.$
 \end{subnumcases}
Since
$$\cos\left[\zeta(1+y)-\frac{\pi}{4}\right]\ge\cos\(\zeta y+\frac{\pi}{4}\)\cos\left[\zeta(1+y)-\frac{\pi}{4}\right]>\cos\left[\zeta(1-y)-\frac{\pi}{4}\right]\cos\(\zeta y-\frac{\pi}{4}\),$$
$\Phi(y)=\cos\left[\zeta(1+y)-\frac{\pi}{4}\right]$.

It is easy to check, just studying the arguments of the cosines involved, that $-\sin(\zeta) k(z,y)\le\frac{1}{2}\le\Phi(y)$, therefore $\sin(\zeta) |k(z,y)|\le\Phi(y)$ for all $z,y\in[-1,1]$.
\end{proof}
We know give a technical lemma that will be used afterwards.
\begin{lem}\label{midpoint} Let $f:[p-c,p+c]\to\bR$ be a symmetric function with respect to $p$, decreasing in $[p,p+c]$. Let $g:[a,b]\to\bR$ be an affine function such that $g([a,b])\subset[p-c,p+c]$. Under these hypothesis, the following hold.
\begin{enumerate}
\item If $g(a)<g(b)<p$ or $p<g(b)<g(a)$ then $f(g(a))<f(g(b))$,
\item if $g(b)<g(a)<p$ or $p<g(a)<g(b)$ then $f(g(a))>f(g(b))$,
\item if $g(a)<p<g(b)$ then $f(g(a))<f(g(b))$ if and only if $g(\frac{a+b}{2})<p$,
\item if $g(b)<p<g(a)$ then $f(g(a))<f(g(b))$ if and only if $g(\frac{a+b}{2})>p$.
\end{enumerate}
\end{lem}
\begin{rem} An analogous result can be established, with the proper changes in the inequalities, if $f$ is increasing in $[p,p+c]$.\end{rem}
\begin{proof}
It is clear that $f(g(a))<f(g(b))$ if and only if $|g(a)-p|>|g(b)-p|$, so $(1)$ and $(2)$ are straightforward. Also, realize that, since $g$ is affine, we have that $g\(\frac{a+b}{2}\)=\frac{g(a)+g(b)}{2}$.\par
Let us prove $(3)$ as $(4)$ is analogous:
$$|g(b)-p|-|g(a)-p|=g(b)-p-(p-g(a))=g(a)+g(b)-2p=2\left[g\(\frac{a+b}{2}\)-p\right].$$
Therefore $|g(a)-p|>|g(b)-p|$ if and only if $g\(\frac{a+b}{2}\)<p$.
\end{proof}
\begin{lem}\label{lemc} Let $\zeta\in(0,\frac{\pi}{4}]$ and $b\ge a\ge0$ such that $a+b=1$. Then $$\sin\(\zeta\)k\(z,y\)\ge c(a)\Phi\(y\)\text{ for }z\in\left[a,b\right],\ y\in\left[-1,1\right],$$ where
$$c(a):=\inf_{y\in[-1,1]}\left\{\frac{\sin(\zeta)\inf\limits_{z\in[a,b]}k(z,y)}{\Phi(y)}\right\}=\frac{[1-\tan(\zeta a)][1-\tan(\zeta b)]}{[1+\tan(\zeta a)][1+\tan(\zeta b)]}.$$
\end{lem}
\begin{proof} We know by Lemma \ref{posstrip} that $k$ is positive in $S_r:=[a,b]\times[-1,1]$. Furthermore, it is proved in \cite{alb-adr} that
$$\frac{\partial k}{\partial t}(t,s)+m\,k(-t,s)=0\sfa t,s\in[-T,T],$$
so, differentiating and doing the proper substitutions we get that
$$\frac{\partial^2 k}{\partial t^2}(t,s)+m^2k(t,s)=0\sfa t,s\in[-T,T].$$
Therefore, $\frac{\partial^2 k}{\partial t^2}<0$ in $S_r$, which means that any minimum of $k$ with respect to $t$ has to be in the boundary of the differentiable regions of $S_r$. Thus, it is clear that, in $S_r$,
$$\sin(\zeta) k(z,y)\ge\eta(z,y):=$$
\begin{equation}
\begin{cases}
\cos([\max\{|\zeta a+\frac{\pi}{4}|,|\zeta b+\frac{\pi}{4}|\})\cos[\zeta(y-1)-\frac{\pi}{4}] , &  |z|<y,\ y\in[b,1],\\
\cos([\max\{|\zeta a+\frac{\pi}{4}|,|\zeta y+\frac{\pi}{4}|\})\cos[\zeta(y-1)-\frac{\pi}{4}] , &  |z|<y,\ y\in[a,b),\\
\cos[\max\{|\zeta(1-y)-\frac{\pi}{4}|,|\zeta(1-b)-\frac{\pi}{4}|]\cos(\zeta y-\frac{\pi}{4}), &  z>|y|,\ y\in[a,b),\\
\cos[\max\{|\zeta(1-a)-\frac{\pi}{4}|,|\zeta(1-b)-\frac{\pi}{4}|]\cos(\zeta y-\frac{\pi}{4}), &  z>|y|,\ y\in[-a,a),\\
\cos[\max\{|\zeta(1-y)-\frac{\pi}{4}|,|\zeta(1-b)-\frac{\pi}{4}|]\cos(\zeta y-\frac{\pi}{4}), &  z>|y|,\ y\in[-b,-a),\\
\cos([\max\{|\zeta a+\frac{\pi}{4}|,|\zeta y+\frac{\pi}{4}|\})\cos[\zeta(1+y)-\frac{\pi}{4}], &  -|z|>y,\ y\in[-b,-a),\\
\cos([\max\{|\zeta a+\frac{\pi}{4}|,|\zeta b+\frac{\pi}{4}|\})\cos[\zeta(1+y)-\frac{\pi}{4}], &  -|z|>y,\ y\in[-1,-b).
\end{cases}
\end{equation}
\par
By definition, $\eta(z,y)\ge\Psi(y):=\sin(\zeta)\,\inf_{z\in[a,b]}k(z,y)$. Also, realize that the arguments of the cosine in (\ref{eqgb2}) are affine functions and that the cosine function is strictly decreasing in $[0,\pi]$ and symmetric with respect to zero. We can apply Lemma \ref{midpoint} to get
\begin{subnumcases}{\label{eta1}\eta\(z,y\)=}
\label{eta1a}\cos\(\zeta b+\frac{\pi}{4}\)\cos\left[\zeta\(y-1\)-\frac{\pi}{4}\right], & $ |z|<y,\ y\in\left[b,1\right],$\\
\label{eta1b}\cos\(\zeta y+\frac{\pi}{4}\)\cos\left[\zeta\(y-1\)-\frac{\pi}{4}\right], & $ |z|<y,\ y\in\left[a,b\),$\\
\label{eta1c}\cos\(\zeta\(1-b\)-\frac{\pi}{4}\)\cos\(\zeta y-\frac{\pi}{4}\)& $\text{if}\quad z>|y|,\ y\in\left[-b,b\),$\\
\label{eta1d}\cos\(\zeta y+\frac{\pi}{4}\)\cos\left[\zeta\(1+y\)-\frac{\pi}{4}\right]& $\text{if}\quad -|z|>y,\ y\in\left[-b,-a\),$\\
\label{eta1e}\cos\(\zeta b+\frac{\pi}{4}\)\cos\left[\zeta\(1+y\)-\frac{\pi}{4}\right]& $\text{if}\quad -|z|>y,\ y\in\left[-1,-b\).$
\end{subnumcases}

\par
Finally, we have to compare the cases \eqref{eta1b} with \eqref{eta1c} for $y\in[a,b)$ and \eqref{eta1d} with \eqref{eta1c} for $y\in[-b,-a)$. Using again Lemma \ref{midpoint}, we obtain the following inequality.
\begin{align*}& \cos\(\zeta\(1-b\)-\frac{\pi}{4}\)\cos\(\zeta y-\frac{\pi}{4}\)-\cos\(\zeta y+\frac{\pi}{4}\)\cos\left[\zeta\(y-1\)-\frac{\pi}{4}\right]\ge\\ & \cos\(\zeta\(1-b\)-\frac{\pi}{4}\)\cos\(\zeta b-\frac{\pi}{4}\)-\cos\(\zeta b+\frac{\pi}{4}\)\cos\left[\zeta\(b-1\)-\frac{\pi}{4}\right]=\sin\zeta>0.\end{align*}
Thus, \eqref{eta1c}$>$\eqref{eta1b} for $y\in[a,b)$.\par
To compare \eqref{eta1d} with \eqref{eta1c} for $y\in[-b,b)$ realize that $k$ is continuous in the diagonal $z=-y$ (see \cite{alb-adr}). Hence, since the expressions of \eqref{eta1d} and \eqref{eta1c} are already locally minimyzing (in their differentiable components) for the variable $z$, it is clear that \eqref{eta1d}$\ge$\eqref{eta1c} for $y\in[-b,-a)$. Therefore,
\begin{subnumcases}{\label{eta2}\Psi\(y\)=}
\label{eta2a}\cos\(\zeta b+\frac{\pi}{4}\)\cos\left[\zeta\(y-1\)-\frac{\pi}{4}\right], & $ y\in\left[b,1\right],$\\
\label{eta2b}\cos\(\zeta y+\frac{\pi}{4}\)\cos\left[\zeta\(y-1\)-\frac{\pi}{4}\right], & $ y\in\left[a,b\),$\\
\label{eta2c}\cos\(\zeta\(1-b\)-\frac{\pi}{4}\)\cos\(\zeta y-\frac{\pi}{4}\)& $\text{if}\quad y\in\left[-b,a\),$\\
\label{eta2e}\cos\(\zeta b+\frac{\pi}{4}\)\cos\left[\zeta\(1+y\)-\frac{\pi}{4}\right]& $\text{if}\quad y\in\left[-1,-b\).$
\end{subnumcases}
It is easy to check that the following order holds:
$$-1<-\frac{\pi}{4\zeta}<-b<\b-1<1-\frac{\pi}{4\zeta}<a<b<\b<1.$$
Thus, we get the following expression
\begin{subnumcases}{\label{quo1}\frac{\Psi\(y\)}{\Phi(y)}=}
\label{quo1a}\cos\(\zeta b+\frac{\pi}{4}\), & $ y\in\left[\b,1\right],$\\
\label{quo1b}\frac{\cos\(\zeta b+\frac{\pi}{4}\)\cos\(\zeta(y-1)-\frac{\pi}{4}\)}{\cos\(\zeta y-\frac{\pi}{4}\)\cos\(\zeta(y-1)+\frac{\pi}{4}\)}, & $ y\in\left[b,\b\),$\\
\label{quo1c}\frac{\cos\(\zeta y+\frac{\pi}{4}\)\cos\(\zeta(y-1)-\frac{\pi}{4}\)}{\cos\(\zeta y-\frac{\pi}{4}\)\cos\(\zeta(y-1)+\frac{\pi}{4}\)}, & $ y\in\left[a,b\),$\\
\label{quo1d}\frac{\cos\(\zeta(1-b)-\frac{\pi}{4}\)}{\cos\(\zeta(y-1)+\frac{\pi}{4}\)}, & $ y\in\left[1-\frac{\pi}{4\zeta},a\),$\\
\label{quo1e}\cos\(\zeta(1-b)-\frac{\pi}{4}\), & $ y\in\left[\b-1,1-\frac{\pi}{4\zeta}\),$\\
\label{quo1f}\frac{\cos\(\zeta(1-b)-\frac{\pi}{4}\)\cos\(\zeta y-\frac{\pi}{4}\)}{\cos\(\zeta y+\frac{\pi}{4}\)\cos\(\zeta(1+y)-\frac{\pi}{4}\)}, & $ y\in\left[-b,\b-1\),$\\
\label{quo1g}\frac{\cos\(\zeta b+\frac{\pi}{4}\)}{\cos\(\zeta y+\frac{\pi}{4}\)}, & $ y\in\left[-\frac{\pi}{4\zeta},-b\),$\\
\label{quo1h}\cos\(\zeta b+\frac{\pi}{4}\), & $ y\in\left[-1,-\frac{\pi}{4\zeta}\).$
\end{subnumcases}\par
To find the infimum of this function we will go through several steps in which we discard different cases. First, it is easy to check the inequalities \eqref{quo1g}$\ge$\eqref{quo1h}$=\eqref{quo1a}$ and \eqref{quo1d}$\ge$\eqref{quo1e}, so we need not to think about \eqref{quo1d}, \eqref{quo1g} and \eqref{quo1h} anymore.\par
Now, realize that $|\zeta(1-b)-\frac{\pi}{4}|\le|\zeta b+\frac{\pi}{4}|\le\pi$. Since the cosine is decreasing in $[0, \pi]$ and symmetric with respect to zero this implies that \eqref{quo1e}$\ge$\eqref{quo1a}.\par
Note that \eqref{quo1c} can be written as
$$g_1(y):=\frac{[1-\tan(\zeta y)](1-\tan[\zeta(1-y)])}{[1+\tan(\zeta y)](1+\tan[\zeta(1-y)])}.$$
Its derivative is
$$g_1'(y)=-\frac{4\zeta[\tan^2(\zeta y)-\tan^2 \zeta(y-1)]}{(\tan \zeta y+1)^2[\tan \zeta(y-1)]^2},$$
which only vanishes at $y=\frac{1}{2}$ for $y\in[a,b]$.
$$g''_1\(\frac{1}{2}\)=-\frac{16\zeta^2\tan\(\frac{\zeta}{2}\)\(\tan^2\frac{\zeta}{2}+1\)}{\(\tan\frac{\zeta}{2}+1\)^4}<0,$$
Therefore $y=\frac{1}{2}$ is a maximum of the function. Since $g_1$ is symmetric with respect to $\frac{1}{2}$ and $a$ is the symmetric point of $b$ with respect to $\frac{1}{2}$, $g(a)=g(b)$ is the infimum of \eqref{quo1c} which is contemplated in \eqref{quo1b} for $y=b$.\par
Making the change of variables $y=\ol y-1$ we have that \eqref{quo1f} can be written as
\begin{equation}\label{quo1f'}\tag{\ref{quo1f}'}\frac{\cos\(\zeta(1-b)-\frac{\pi}{4}\)\cos\(\zeta(\ol y-1)-\frac{\pi}{4}\)}{\cos\(\zeta \ol y-\frac{\pi}{4}\)\cos\(\zeta(\ol y-1)+\frac{\pi}{4}\)}\quad\text{if}\quad \ol y\in\left[a,\b\).\end{equation}
Since \eqref{quo1e}$\ge$\eqref{quo1a}, it is clear now that \eqref{quo1f'}$\ge$\eqref{quo1b} in $[b,\b)$.\par
Let
$$g_2(y):=\frac{\cos\(\zeta(y-1)-\frac{\pi}{4}\)}{\cos\(\zeta y-\frac{\pi}{4}\)\cos\(\zeta(y-1)+\frac{\pi}{4}\)}.$$
Then
$$g_2'(y)=\frac{\zeta}{4}\cdot\frac{\sin\left[\zeta(2-y)-\frac{\pi}{4}\right]+\sin\left[\zeta(3y-2)-\frac{\pi}{4}\right]+4\cos\left[\zeta y-\frac{\pi}{4}\right]}{\sin^2\left[\zeta y+\frac{\pi}{4}\right]\cos^2\left[\zeta(1-y)-\frac{\pi}{4}\right]^2}.$$
Since the argument in the cosine of the numerator is in the interval $[-\frac{\pi}{4},\frac{\pi}{4}]$ for $y\in[a,1]$, it is clear that $g_2'(y)>0$ for $y\in[a,1]$, which implies that $g_2$ is increasing  in that interval and \eqref{quo1b} and \eqref{quo1f} reach their infimum in the left extreme point of their intervals of definition.\par
We have then that
$$c(a)=\inf\limits_{y\in[-1,1]}\frac{\Psi(y)}{\Phi(y)}=\min\left\{\cos\(\zeta b+\frac{\pi}{4}\),\frac{\cos\(\zeta b+\frac{\pi}{4}\)\cos\(\zeta(b-1)-\frac{\pi}{4}\)}{\cos\(\zeta b-\frac{\pi}{4}\)\cos\(\zeta(b-1)+\frac{\pi}{4}\)},\frac{\cos\(-\zeta b-\frac{\pi}{4}\)}{\cos\(-\zeta b+\frac{\pi}{4}\)}\right\}.$$
The third element of the set is clearly greater or equal than the first. The second element is $\cos\(\zeta b+\frac{\pi}{4}\)g_2(b)$. Since $g_2$ is increasing in $[a,1]$,
$$\cos\(\zeta b+\frac{\pi}{4}\)g_2(b)\le\cos\(\zeta b+\frac{\pi}{4}\)g_2(1)=\cos\(\zeta b+\frac{\pi}{4}\)\frac{\cos(\zeta)}{\sin(\zeta)}\le\cos\(\zeta b+\frac{\pi}{4}\).$$
Therefore,
$$c(a)=\frac{\cos\(\zeta b+\frac{\pi}{4}\)\cos\(\zeta(b-1)-\frac{\pi}{4}\)}{\cos\(\zeta b-\frac{\pi}{4}\)\cos\(\zeta(b-1)+\frac{\pi}{4}\)}=\frac{[1-\tan(\zeta a)][1-\tan(\zeta b)]}{[1+\tan(\zeta a)][1+\tan(\zeta b)]}.$$
\end{proof}
\begin{rem}It is easy to find an upper estimate of $c(a)$. Just assume $a=b=\frac{1}{2}$.
$$c(a)\le c(0)=\(\frac{1-\tan\frac{\zeta}{2}}{1+\tan\frac{\zeta}{2}}\)^2\le\(\frac{1-\tan\frac{\pi}{8}}{1+\tan\frac{\pi}{8}}\)^2=\frac{(2-\sqrt 2)^2}{2}= 0.17157\dots$$
\end{rem}

\par We can do the same study for $\zeta\in(0,\frac{\pi}{4}]$. The proofs are almost the same, but in this case the calculations are much easier.
\begin{lem}If $\zeta\in(0,\frac{\pi}{4}]$ then
$\sin(\zeta) |k(z,y)|\le\Phi(y):=\max_{z\in[-1,1]}k(z,y)$ where $\Phi$ admits the following expression:
$$\Phi(y)=\begin{cases}
\cos\left[\zeta(y-1)+\frac{\pi}{4}\right]\cos\(\zeta y-\frac{\pi}{4}\) , &  y\in\left[0,1\right] \\
\cos\(\zeta y+\frac{\pi}{4}\)\cos\left[\zeta(y+1)-\frac{\pi}{4}\right] , &  y\in[-1,0) \\
\end{cases}
$$
\end{lem}
\begin {proof}
This time, a simplified version of inequality \eqref{firstineq} holds,

\begin{equation}\label{firstineq2}
\sin(\zeta) k(z,y)\le\xi(z,y):=
\begin{cases}
\cos[\zeta(1-|y|)-\frac{\pi}{4}]\cos(\zeta y-\frac{\pi}{4}), &  z>|y|, \\
\cos(\zeta y-\frac{\pi}{4})\cos[\zeta(y-1)-\frac{\pi}{4}] , &  |z|<y, \\
\cos(\zeta y+\frac{\pi}{4})\cos[\zeta(1+y)-\frac{\pi}{4}], &  -|z|>y, \\
\frac{\sqrt 2}{2}\cos(\zeta y-\frac{\pi}{4}), &  z<-|y|,
\end{cases}
\end{equation}

so we only need to study two cases. If $y>0$, we are in the same situation as in the case $y\in[1-\frac{\pi}{4\zeta},\frac{\pi}{4\zeta})$ studied in Lemma \ref{max1}. Hence, $\Phi(y)=\cos\left[\zeta(y-1)+\frac{\pi}{4}\right]\cos\(\zeta y-\frac{\pi}{4}\)$. If $y<0$  we are in the same situation as in the case
$y\in[-\frac{\pi}{4\zeta},\frac{\pi}{4\zeta}-1)$. Therefore, \[\Phi(y)=\cos\(\zeta y+\frac{\pi}{4}\)\cos\left[\zeta(y+1)-\frac{\pi}{4}\right].\]
\end{proof}
\begin{lem} Let $\zeta\in(0,\frac{\pi}{4}]$ and $b\ge a\ge0$ such that $a+b=1$. Then $$\sin\(\zeta\)k\(z,y\)\ge c(a)\Phi\(y\)\text{ for }z\in\left[a,b\right],\ y\in\left[-1,1\right],$$ where
$$c(a):=\inf_{y\in[-1,1]}\left\{\frac{\sin(\zeta)\inf\limits_{z\in[a,b]}k(z,y)}{\Phi(y)}\right\}=\frac{[1-\tan(\zeta a)][1-\tan(\zeta b)]}{[1+\tan(\zeta a)][1+\tan(\zeta b)]}.$$
\end{lem}
\begin{proof}
$\Psi$ is as in \ref{eta2}, but we get the simpler expression
\begin{subnumcases}{\label{quo12}\frac{\Psi\(y\)}{\Phi(y)}=}
\label{quo1b2}\frac{\cos\(\zeta b+\frac{\pi}{4}\)\cos\(\zeta(y-1)-\frac{\pi}{4}\)}{\cos\(\zeta y-\frac{\pi}{4}\)\cos\(\zeta(y-1)+\frac{\pi}{4}\)}, & $ y\in\left[b,1\right],$\\
\label{quo1c2}\frac{\cos\(\zeta y+\frac{\pi}{4}\)\cos\(\zeta(y-1)-\frac{\pi}{4}\)}{\cos\(\zeta y-\frac{\pi}{4}\)\cos\(\zeta(y-1)+\frac{\pi}{4}\)}, & $ y\in\left[a,b\),$\\
\label{quo1d2}\frac{\cos\(\zeta(1-b)-\frac{\pi}{4}\)}{\cos\(\zeta(y-1)+\frac{\pi}{4}\)}, & $ y\in\left[0,a\),$\\
\label{quo1f2}\frac{\cos\(\zeta(1-b)-\frac{\pi}{4}\)\cos\(\zeta y-\frac{\pi}{4}\)}{\cos\(\zeta y+\frac{\pi}{4}\)\cos\(\zeta(1+y)-\frac{\pi}{4}\)}, & $ y\in\left[-b,0\),$\\
\label{quo1h2}\cos\(\zeta b+\frac{\pi}{4}\), & $ y\in\left[-1,-b\).$
\end{subnumcases}\par
By the same kind of arguments used in the proof of Lemma \ref{lemc}, we get the desired result.

\end{proof}

\begin{lem}\label{intabs}
$$\sup_{t\in \left[-T,T\right]}\int_{-T}^{T}|k\(t,s\)|ds=\begin{cases} \frac{1}{\omega} &\text{if } \zeta\in\(0,\frac{\pi}{4}\right], \\
\frac{1}{\omega}\left[1+\frac{\sqrt2\cos\frac{2\zeta+\pi}{3}\sin\frac{\pi-4\zeta}{12}+\cos\frac{\pi-\zeta}{3}\(1-\sin\frac{2\zeta+\pi}{3}\)}{\sin\zeta}\right], &  \zeta\in\left[\frac{\pi}{4},\frac{\pi}{2}\right].\end{cases}$$
\end{lem}
\begin{proof}
First of all, if $\zeta\in\left[0,\frac{\pi}{4}\right]$, then $|k\(t,s\)|=k\(t,s\)$. The solution of the problem $x'\(t\)+\omega\,x\(-t\)=1$, $x\(-T\)=x\(T\)$ is clearly $u\(t\)\equiv\frac{1}{\omega}$, but at the same time it has to be of the kind in equation \eqref{gensol}, so $u\(t\)=\int_{-T}^{T}k\(t,s\)ds$. This proves the first part.\par
If $\zeta\in\left[\frac{\pi}{4},\frac{\pi}{2}\right]$, then $$\int_{-T}^{T}|k\(t,s\)|ds=\int_{-T}^{T}k^+\(t,s\)ds+\int_{-T}^{T}k^-\(t,s\)ds=\frac{1}{\omega}+2\int_{-T}^{T}k^-\(t,s\)ds.$$ We make two observations here.\par
From equation \eqref{gbarra}, it is easy to check that $k\(t+T,s+T\)=k\(t,s\)$ and \[k\(t+T,s\)=k\(t,s+T\),\] for a.e. $t,s\in\left[-T,0\right]$. Hence, for $t\in\left[-T,0\right]$ and a function $\xi:\bR\to\bR$, using the change of variables $r=s+T$, $\tau=s-T$, we have that
\begin{align*} \int_{-T}^{T}\xi\(k\(t+T,s\)\)ds & =\int_{-T}^{0}\xi\(k\(t+T,s\)\)ds+\int_{0}^{T}\xi\(k\(t+T,s\)\)ds \\ &=\int_{-T}^{0}\xi\(k\(t,s+T\)\)ds +\int_{-T}^{0}\xi\(k\(t+T,\tau+T\)\)d\tau\\ & =\int_{0}^{T}\xi\(k\(t,r\)\)dr+\int_{-T}^{0}\xi\(k\(t,\tau\)\)d\tau=\int_{-T}^{T}\xi\(k\(t,s\)\)ds.\end{align*}
Therefore, $\sup_{t\in \left[-T,T\right]}\int_{-T}^{T}|k\(t,s\)|ds=\sup_{t\in \left[-T,0\right]}\int_{-T}^{T}|k\(t,s\)|ds$.
The second observation is that, taking into account Lemma \ref{posstrip}, $k\(t,s\)$ is positive in $\(-\frac{\pi}{4\omega },1-\frac{\pi}{4\omega }\)\times\left[-1,1\right]$, so $$\sup_{t\in \left[-T,0\right]}\int_{-T}^{T}|k\(t,s\)|ds=\sup_{t\in \left[-T,0\right]\backslash\(-\frac{\pi}{4\omega },1-\frac{\pi}{4\omega }\)}\int_{-T}^{T}|k\(t,s\)|ds.$$ Using the same kind of arguments as in Lemma \ref{posstrip}, it is easy to check that $k\(t,s\)$ is negative in $\(-T,-\frac{\pi}{4\omega }\)\times\(t,-\frac{\pi}{4\omega }\)$ if $t\in\(-T,-\frac{\pi}{4\omega }\)$ and $\(\frac{\pi}{4\omega }-1,0\)\times\(t,1-\frac{\pi}{4\omega }\)$ if $t\in\(\frac{\pi}{4\omega }-1,0\)$, so it is enough to compute $\eta\(t\):=\int_{B\(t\)}k^-\(t,s\)ds$ where $B\(t\)=\{s\in\left[-T,T\right]\ :\ \(t,s\)\in\supp\(k^-\)\}$.
$$2\omega\sin\(\zeta\)\eta\(t\)=\begin{cases} \cos\(\omega t+\zeta+\frac{\pi}{4}\)\left[1+\sin\(\omega t-\frac{\pi}{4}\)\right], & t\in\(-T,-\frac{\pi}{4\omega }\), \\ \sqrt 2\cos\(\omega t+\zeta+\frac{\pi}{4}\)\sin\omega t+\cos\(\omega t+\frac{\pi}{4}\)\left[1-\sin\(\omega t+\zeta+\frac{\pi}{4}\)\right], & t\in\(\frac{\pi}{4\omega }-1,0\).\end{cases}$$
With the change of variable $t=zT$,
$$2\omega\sin\(\zeta\)\eta\(z\)=\begin{cases} \eta_1\(z\) &\text{if }z\in\(-1,-\frac{\pi}{4\zeta }\), \\ \eta_2\(z\), & z\in\(\frac{\pi}{4\zeta }-1,0\).\end{cases}$$
where
$$\eta_1\(z\)=\cos\left[\zeta\(z+1\)+\frac{\pi}{4}\right]\left[1+\sin\(\zeta z-\frac{\pi}{4}\)\right]$$ and
$$\eta_2\(z\)=\sqrt 2\cos\left[\zeta\(z+1\)+\frac{\pi}{4}\right]\sin\zeta z+\cos\(\zeta z+\frac{\pi}{4}\)\left[1-\sin\(\zeta\(z+1\)+\frac{\pi}{4}\)\right].$$
It is easy to check that
$$\eta_1'\(-1\)\le0,\ \eta_1'\(-\frac{\pi}{4\zeta }\)=0,\ \eta_1''\(z\)\ge0\text{ for }z\in\left[-1,-\frac{\pi}{4\zeta }\right],$$
$$\eta_1'\(-1\)=\eta_2\(0\),$$
$$\eta_2'\(\frac{\pi}{4\omega }-1\)>0,\ \eta_2'\(0\)<0,\ \eta_2''\(z\)\ge0\text{ for }z\in\left[\frac{\pi}{4\zeta }-1,0\right].$$
With these facts we conclude that there is a unique maximum of the function $\eta\(z\)$ in the interval $\(\frac{\pi}{4\zeta }-1,0\)$, precisely where $\eta_2'\(z\)=\zeta \(\cos\left[\zeta\( 1+ 2 z \)\right] - \sin\(\frac{\pi}{4} + z \zeta\)\)=0$, this is, for $z=\frac{1}{3}(\frac{\pi}{4}-1)$, and therefore the statement of the theorem holds.
\end{proof}
\begin{lem}\label{leminf}Let $\omega\in\left[\frac{\pi}{4}T,\frac{\pi}{2}T\right]$ and $T-\frac{\pi}{4\omega}<a< b=T-a<\frac{\pi}{4\omega}$. Then
$$2\omega\sin(\zeta)\inf_{t\in[a,b]}\int_{a}^{b}k(t,s)\,ds=\sin\omega(T-2a)+\cos\zeta-\cos2\omega a.$$
\end{lem}

\begin{proof}
It is easy to check that
$$2\omega\sin(\zeta)\int_{-T}^{s}k(t,r)\,dr=
\begin{cases}
\sin\omega(T+s+t)-\cos\omega(T+s-t)-\sin\omega t+\cos\omega t, &  |t|\le-s, \\
\sin\omega(T+s+t)-\cos\omega(T-s+t)-\sin\omega t+\cos\omega t, &  |s|\le-t, \\
-\sin\omega(T-s-t)-\cos\omega(T+s-t)-\sin\omega t+\cos\omega t +2\sin\omega t, &  |s|\le t, \\
-\sin\omega(T-s-t)-\cos\omega(T-s+t)-\sin\omega t+\cos\omega t +2\sin\omega t, &  |t|\le s. \\
\end{cases}
$$
Therefore $\int_{a}^{b}k(t,s)\,ds=\int_{-T}^{b}k(t,s)\,ds-\int_{-T}^{a}k(t,s)\,ds$, this is,
$$2\omega\sin(\zeta)\int_{a}^{b}k(t,s)\,ds=\sin\omega(T-a-t)-\sin\omega(a-t)+\cos\omega(T+a-t)-\cos\omega(a+t),  \text{ if } t\in[a,b].$$
Using similar arguments to the ones used in the proof of Lemma \ref{lemc} we can show that
$$2\omega\sin(\zeta)\inf_{t\in[a,b]}\int_{a}^{b}k(t,s)\,ds=\sin\omega(T-2a)+\cos\zeta-\cos2\omega a.$$
\end{proof}
With the same method, we can prove the following corollary.
\begin{cor}Let $\omega\in\(0,\frac{\pi}{4}T\right]$ and $0<a< b=T-a<1$. Then
$$2\omega\sin(\zeta)\inf_{t\in[a,b]}\int_{a}^{b}k(t,s)\,ds=\sin\omega(T-2a)+\cos\zeta-\cos2\omega a.$$
\end{cor}
\begin{rem}If $\omega\in\(0, \frac{\pi}{4}T\right]$, then
$$\inf_{t\in[-T,T]}\int_{-T}^{T}k(t,s)\,ds=\frac{1}{\omega},$$
just because of the observation in the proof of Lemma \ref{intabs}.
\end{rem}

Now we can state conditions $(I^0_\rho)$ and $(I^1_\rho)$ for the special case of problem \eqref{eqgenpro2}-\eqref{prdic}:

\begin{enumerate}
\item[$(\mathrm{I}_{\protect\rho,\omega}^{1})$] \label{EqB2} Let
$$
  f^{{-\rho},{\rho}}_\omega:=\sup \left\{\frac{h(t,u,v)+\omega v}{\rho }:\;(t,u,v)\in
[ -T,T]\times [ -\rho,\rho ]\times [-\rho,\rho ]\right\}.$$
\hspace{2em}There exist $\rho> 0$ and $\omega\in\(0,\frac{\pi}{4}\right]$ such that $f^{-\rho,\rho}_\omega <\omega$,\\
OR \\
\hspace*{2em} there exist $\rho> 0$ and $\omega\in\(\frac{\pi}{4},\frac{\pi}{2}\right]$ such that

$$
 f^{-\rho,\rho}_\omega\cdot\left[1+\frac{\sqrt2\cos\frac{2\zeta+\pi}{3}\sin\frac{\pi-4\zeta}{12}+\cos\frac{\pi-\zeta}{3}\(1-\sin\frac{2\zeta+\pi}{3}\)}{\sin\zeta}\right]
 <\omega.
$$

\item[$(\mathrm{I}_{\protect\rho,\omega}^{0})$] there exist $\rho >0$ such that
such that
$$
f_{(\rho ,{\rho /c})}^\omega\cdot\inf_{t\in [a,b]}\int_{a}^{b}k(t,s)\,ds>1,
$$
where
$$
f_{(\rho ,{\rho /c})}^\omega =\inf \left\{\frac{h(t,u,v)+\omega v}{\rho }%
:\;(t,u,v)\in [a,b]\times [\rho ,\rho /c]\times
[-\rho /c,\rho /c]\right\}.$$
\end{enumerate}

\begin{thm}

\label{thmmsol2} Let $\omega\in\(0,\frac{\pi}{2}T\right]$. Let $[a,b]\subset[-T,T]$ such that $a=1-b\in(\max\{0,T-\frac{\pi}{4\omega}\},\frac{T}{2})$. Let
$$c=\frac{[1-\tan(\omega a)][1-\tan(\omega b)]}{[1+\tan(\omega a)][1+\tan(\omega b)]}.$$

Problem \eqref{eqgenpro2}-\eqref{prdic} has at least one non-zero solution
in $K$ if either of the following conditions hold.

\begin{enumerate}

\item[$(S_{1})$] There exist $\rho _{1},\rho _{2}\in (0,\infty )$ with $\rho
_{1}/c<\rho _{2}$ such that $(\mathrm{I}_{\rho _{1},\omega}^{0})$ and $(\mathrm{I}_{\rho _{2},\omega}^{1})$ hold.

\item[$(S_{2})$] There exist $\rho _{1},\rho _{2}\in (0,\infty )$ with $\rho
_{1}<\rho _{2}$ such that $(\mathrm{I}_{\rho _{1},\omega}^{1})$ and $(\mathrm{I}%
_{\rho _{2},\omega}^{0})$ hold.
\end{enumerate}
The integral equation \eqref{eqhamm} has at least two non-zero solutions in $K$ if one of
the following conditions hold.

\begin{enumerate}

\item[$(S_{3})$] There exist $\rho _{1},\rho _{2},\rho _{3}\in (0,\infty )$
with $\rho _{1}/c<\rho _{2}<\rho _{3}$ such that $(\mathrm{I}_{\rho
_{1},\omega}^{0}),$ $(
\mathrm{I}_{\rho _{2},\omega}^{1})$ $\text{and}\;\;(\mathrm{I}_{\rho _{3},\omega}^{0})$
hold.

\item[$(S_{4})$] There exist $\rho _{1},\rho _{2},\rho _{3}\in (0,\infty )$
with $\rho _{1}<\rho _{2}$ and $\rho _{2}/c<\rho _{3}$ such that $(\mathrm{I}%
_{\rho _{1},\omega}^{1}),\;\;(\mathrm{I}_{\rho _{2},\omega}^{0})$ $\text{and}\;\;(\mathrm{I%
}_{\rho _{3},\omega}^{1})$ hold.
\end{enumerate}
The integral equation \eqref{eqhamm} has at least three non-zero solutions in $K$ if one
of the following conditions hold.

\begin{enumerate}

\item[$(S_{5})$] There exist $\rho _{1},\rho _{2},\rho _{3},\rho _{4}\in
(0,\infty )$ with $\rho _{1}/c<\rho _{2}<\rho _{3}$ and $\rho _{3}/c<\rho
_{4}$ such that $(\mathrm{I}_{\rho _{1},\omega}^{0}),$ $(\mathrm{I}_{\rho _{2},\omega}^{1}),\;\;(\mathrm{I}%
_{\rho _{3},\omega}^{0})\;\;\text{and}\;\;(\mathrm{I}_{\rho _{4},\omega}^{1})$ hold.

\item[$(S_{6})$] There exist $\rho _{1},\rho _{2},\rho _{3},\rho _{4}\in
(0,\infty )$ with $\rho _{1}<\rho _{2}$ and $\rho _{2}/c<\rho _{3}<\rho _{4}$
such that $(\mathrm{I}_{\rho _{1},\omega}^{1})$, $(\mathrm{I}_{\rho
_{2},\omega}^{0})$, $(\mathrm{I}_{\rho _{3},\omega}^{1})$ and $(\mathrm{I}%
_{\rho _{4},\omega}^{0})$ hold.
\end{enumerate}
\end{thm}
\subsection{Example}
Consider problem \eqref{eqgenpro2}-\eqref{prdic} with
$$h(t,u,v)=\frac{1}{2+(t-1)^2}+\frac{u^2}{5}+2u+\frac{1}{1+7v^2}+7.$$
Let $T=1$, $\zeta=\omega=1.5$, $a=.48$, $b=.52$, $\rho_1=1$, $\rho_2=2$. Conditions ($C_1$)--($C_3$) are clearly satisfied by the results proved before. $(C_4)$ is implied in a straightforward way from the expression of $h$, so we are in the hypothesis of Theorem \ref{thmk}. Also,
\begin{align*}
c & =0.000353538\dots, \\
r_1: & =\omega\left[1+\frac{\sqrt2\cos\frac{2\zeta+\pi}{3}\sin\frac{\pi-4\zeta}{12}+\cos\frac{\pi-\zeta}{3}\(1-\sin\frac{2\zeta+\pi}{3}\)}{\sin\zeta}\right]^{-1}=11.5009\dots,\\
r_2: & =\(\inf_{t\in[a,b]}\int_{a}^{b}k(t,s)\,ds\)^{-1}=\(\frac{\sin\omega(T-2a)+\cos\zeta-\cos2\omega a}{2\omega\sin\zeta}\)^{-1}=6.58486\dots,\\
f^{-\rho_1,\rho_1}_\omega & =\frac{h(1,\rho_1,\rho_1)+\rho_1\omega}{\rho_1}=11.325,
\\
f_{(\rho_2 ,{\rho_2 /c})}^\omega & =\frac{h(a,\rho_2,0)}{\rho_2}=6.62418\dots
\end{align*}
Clearly, $f^{-\rho_1,\rho_1}_\omega<r_1$ and $f_{(\rho_2 ,{\rho_2 /c})}^\omega>r_2$, so condition $(S_{2})$ in the previous theorem is satisfied, and therefore problem \eqref{eqgenpro2}-\eqref{prdic} has at least one solution.

\end{document}